\documentclass[psamsfonts,12pt]{amsart}

\usepackage[margin=1in]{geometry}  
\usepackage{graphicx}              
\usepackage{amsmath}               
\usepackage{amsfonts}              
\usepackage{amsthm}                
\usepackage{verbatim}              
\usepackage{indentfirst}           
\usepackage{underscore}
\usepackage{enumitem}
\usepackage{mathtools} 
\usepackage{amssymb}
\usepackage[all]{xy}
\usepackage[hidelinks=true]{hyperref}
\usepackage{cleveref}
\crefname{section}{§}{§§}

\newtheorem{thm}{Theorem}[section]
\newtheorem{lem}[thm]{Lemma}
\newtheorem{prop}[thm]{Proposition}
\newtheorem*{definition*}{Definition}
\newtheorem{cor}[thm]{Corollary}
\newtheorem{conj}[thm]{Conjecture}

\theoremstyle{remark}
       \newtheorem*{rmk}{Remark}
\theoremstyle{remark}

\newcommand{\llra}{\Longleftrightarrow}

\newcommand{\Z}{\mathbb{Z}}
\newcommand{\Q}{\mathbb{Q}}
\newcommand{\fr}{\frac}
\newcommand{\tu}{\textup}

\newcommand{\Mod}[1]{\ (\textup{mod}\ #1)}

\newcommand{\Gal}{Gal}

\usepackage{amssymb,fge}

\title{Dynamical Galois Groups of Trinomials and Odoni's Conjecture}
\date{\today}

\author[Nicole R. Looper]{Nicole R. Looper}
\address{Department of Mathematics, Northwestern University; 2033 Sheridan Road, Evanston, IL 60208, USA}
\email{nlooper@math.northwestern.edu}

\begin{document} 

\maketitle

\begin{abstract} \normalsize We prove Odoni's conjecture in all prime degrees; namely, we prove that for every positive prime $p$, there exists a degree $p$ polynomial $\varphi\in\Z[x]$ with surjective arboreal Galois representation. We also show that Vojta's conjecture implies the existence of such a polynomial in many degrees $d$ which are not prime.

\end{abstract}
\renewcommand{\thefootnote}{}

\section{Introduction}

Let $K$ be a number field, and fix $\varphi(x)\in K[x]$. Let $\varphi^n$ denote the $n$th iterate of $\varphi$. The pre-image tree associated to $\varphi$, with root 0, has as its set of vertices \[T_{\infty}=\bigsqcup_{n\ge 0} \varphi^{-n}(0).\] (Note that by convention, $\varphi^0(0)=\{0\}$.) Given two vertices $\alpha\in\varphi^{-n}(0)$ and $\beta\in \varphi^{-(n+1)}(0)$, an edge connects $\alpha,\beta$ if and only if $\varphi(\beta)=\alpha$. Assume the $\varphi^n$ are all separable, so that $0$ is not contained in the forward orbit of any critical point of $\varphi$. Then the group $\tu{Aut}(T_{\infty})$ of graph automorphisms of $T_{\infty}$ is isomorphic to the infinite iterated wreath product of the symmetric group on $d$ elements \cite{Nekra}. The absolute Galois group $\Gal(\overline{K}/K)$ acts naturally on $T_{\infty}$ by its action on the $\varphi^{-n}(0)$, inducing a tree automorphism; one thereby obtains a representation \[\rho: \Gal(\overline{K}/K)\to \tu{Aut}(T_{\infty}).\]  Its image will be denoted by $G_{K}(\varphi)$.

We investigate the question of when the representation $\rho$ is surjective and when the image $G_K(\varphi)$ has finite index in $\tu{Aut}(T_{\infty})$. (See Question 1.1 of \cite{Jones}.) For $d=2$, this question has been studied in some depth. Assuming the $abc$ conjecture over $\Q$ along with an irreducibility condition, a quadratic polynomial $\varphi\in \Z[x]$ with an infinite critical orbit has $[\tu{Aut}(T_{\infty}):G_{\Q}(\varphi)]$ finite \cite{Khoa,Jones}. For analogous conditions on $\varphi$ over a one-variable function field of characteristic zero, a finite index result holds unconditionally \cite{Wade}. 

In the higher degree case, far less is known.  In 1985, R.W.K. Odoni set forth the following conjecture: \begin{conj}[Odoni]{\label{conj:Odoni}} For every $d\ge 2$, there exists a polynomial $\varphi(x)\in\Z[x]$ of degree $d$ such that $G_{\Q}(\varphi)= \tu{Aut}(T_{\infty})$. \end{conj} (See Conjecture 2.2 of the survey by Jones \cite{Jones}, and Conjecture 7.5 of \cite{Odoni} for a more general formulation.) Odoni himself proved his conjecture in the case $d=2$, by showing that the polynomial $\varphi(x)=x^2-x+1$ yields $G_{\Q}(\varphi)=\tu{Aut}(T_{\infty})$ \cite{Odoni2}. He also proved that a certain `generic' degree $d\ge 2$ polynomial over a field $K$ of characteristic zero has surjective arboreal representation \cite{Odoni}. If $K$ is a number field, Hilbert's irreducibility theorem then implies that $\Gal(\varphi^n)/K$ surjects onto an appropriate truncated subtree of $T_{\infty}$ for all but a thin set $E_n$ of degree $d$ polynomials $\varphi$ defined over $K$. However, it is a priori possible that the union of the $E_n$ is the set of all degree $d$ polynomials over $K$; stated otherwise, it is possible that for every degree $d$ polynomial $\varphi(x)\in K[x]$, one obtains $G_{K}(\varphi)\ne \tu{Aut}(T_{\infty})$. 

In this article, we first show that Odoni's conjecture is true for all prime values of $d$. We accomplish this by examining, for each degree $d$, a family of polynomials having 0 as a strictly preperiodic critical point of multiplicity $d-2$, and a nonzero critical point of multiplicity one with infinite forward orbit. This type of critical orbit behavior is similar to that of the quadratic rational function appearing in Theorem 1.2 of \cite{Jones2}.

\begin{thm}{\label{thm:primeOdoni}} Let $\phi_{(p,k)}(x)=x^p+kpx^{p-1}-kp$, where $k\in\Z^+$, $p\ge 3$ is prime, and $p$ does not divide $k$. If $p\ge 5$, and $k\equiv 1\Mod{3}$, then $G_{\Q}(\phi_{(p,k)})=\tu{Aut}(T_{\infty})$. Moreover, $G_{\Q}(\phi_{(3,2)})=\tu{Aut}(T_{\infty})$. Hence Odoni's conjecture holds in all prime degrees. \end{thm}

More generally, for each prime $p$, we produce a family of degree $p$ polynomials for which a finite index result holds.

\begin{thm}{\label{thm:primefinindex}} Let $\phi_{(p,k)}(x)=x^p+kpx^{p-1}-kp$, where $k\in\Z^+$, $p\ge 3$ is prime, and $p$ does not divide $k$. Suppose that the critical point $a=-k(p-1)$ has infinite forward orbit. Then $G_{\Q}(\phi_{(p,k)})$ has finite index in $\tu{Aut}(T_{\infty})$. \end{thm}

We then prove that assuming Vojta's conjecture (Conjecture \ref{conj:Vojta}), Odoni's conjecture holds for many other values of $d=\tu{deg}(\varphi)$. Here, we use a family of polynomials $\phi_{(d,c)}$ for each degree $d$, having 0 as a critical point of multiplicity $d-2$ that maps to the other finite critical point under $\phi_{(d,c)}$. 

\begin{thm}{\label{thm:abcOdoni}} Assume the Vojta (or Hall-Lang) conjecture holds. Let $\phi_{(d,c)}(x)=x^d-cdx^{d-1}+c(d-1)$, where $c\in\Z^+$, $d\ge 3$.  Suppose that either:
\begin{enumerate}[topsep=5pt, partopsep=5pt, itemsep=2pt]
\item[\textup{(1)}] $d$ is prime

\end{enumerate}
or 

\begin{enumerate}[topsep=5pt, partopsep=5pt, itemsep=2pt]

\item[\textup{(2)}] there exists a prime $q$ dividing $d-1$ such that $(d-1,v_q(d-1))=1$. 
\end{enumerate} Let $ \mathcal{B}_d$ be the set of $c\in\Z^+$ such that $\phi_{(d,c)}^n(x)$ is irreducible for all $n\ge 1$, and $(c,d-1)=1$. If $c\in  \mathcal{B}_d$, then $G_{\Q}(\phi_{(d,c)})$ has finite index in $\tu{Aut}(T_{\infty})$. Moreover, for all but finitely many $c\in  \mathcal{B}_d$, we have $G_{\Q}(\phi_{(d,c)})=\tu{Aut}(T_{\infty})$.\end{thm}

Although the irreducibility of the iterates $\phi_{(d,c)}^n$ in Theorem \ref{thm:abcOdoni} may appear to be a strong assumption, it holds, for instance, when $c$ is a prime not dividing $d-1$, as the $\phi_{(d,c)}^n(x)$ are then all Eisenstein at $c$. Odoni's conjecture for degrees $d$ satisfying the hypothesis of Theorem \ref{thm:abcOdoni} follows by observing that all sufficiently large prime values of $c\in\Z^+$ lie in $  \mathcal{B}_d$.

Theorem \ref{thm:primeOdoni} yields a corollary about the density of prime divisors appearing in sequences given by the action of $\phi_{(p,k)}$. 

\begin{cor}{\label{cor:density}} Let $a_0\in \Q$ have infinite forward orbit under $\phi_{(p,k)}$, and for a prime $p\in\Z$, let $v_p$ denote the $p$-adic valuation. Let $\phi_{(p,k)}$ satisfy the hypotheses of Theorem \ref{thm:primeOdoni}, so that $G_{\Q}(\phi_{(p,k)})=\tu{Aut}(T_{\infty})$, and let \[P(a_0)=\{\tu{primes } p\in\Z^+: v_p(\phi_{(p,k)}^i(a_0))>0 \tu{ for at least one } i\ge 0\}. \] Then $P(a_0)$ has natural density zero in the set of all positive primes in $\Z$.\end{cor}

Finally, using techniques from the proof of Theorem \ref{thm:primeOdoni}, we exhibit a polynomial $\varphi(x) \in\Z[x]$ such that $G_{\Q}(\varphi)\le \tu{Aut}(T_{\infty})$ has index 2.  This is the first known example of a polynomial whose arboreal representation over $\Q$ is of finite index, yet is not surjective.

 \begin{thm}{\label{thm:index2}} Let $\phi(x)=x^3+7x^2-7$. Then $G_{\Q}(\phi)$ is an index 2 subgroup of $\tu{Aut}(T_{\infty})$.\end{thm}

It is also worth remarking that the topic of Galois groups of trinomials over $\Q$ has been extensively studied (see for example \cite{Salinier1}, \cite{Salinier2},\cite{Salinier3},and \cite{Osada}). Hilbert used trinomials to prove the existence, for each degree $d$, of a polynomial over $\Q$ whose Galois group is $S_d$ \cite{Hilbert}. Since then, some constructive results have been shown, concerning which trinomials over $\Q$ have this property. Moreover, the generic trinomial $x^d+tx^s+w\in \Q(t,w)[x]$, where $d>s\ge 1$, is known to have Galois group $S_d$ over $\Q(t,w)$ as long as $(d,s)=1$ \cite{Smith}. Here we begin to extend this question to iterated Galois groups, and to the infinite Galois groups $G_{\Q}(\varphi)$.
\newline

\indent \textbf{Acknowledgements:} The author would like to thank Nigel Boston, Laura DeMarco, Wade Hindes, Rafe Jones, Jamie Juul, and Tom Tucker for helpful conversations related to this work.  Several of these took place at the May 2016 AIM workshop titled ``The Galois theory of orbits in arithmetic dynamics,'' and the author also thanks AIM and the organizers of this workshop.

\section{Ramification Behavior in Extensions Defined by Trinomials}\label{section:ram}

In this section, we prove that certain irreducible monic trinomials $f(x)$ over a number field $F$ have the property that primes ramifying in $\Gal(f(x))/F$ often have inertia groups of order 2. This result is an adaptation of Theorem 2 of \cite{LNV} to relative extensions of number fields (compare p.151 of \cite{FrohlichTaylor}).

\begin{thm}{\label{thm:LNV}}

Let $F$ be a number field, with ring of integers $\mathcal{O}_F$. Let $f(x)=x^d+Ax^s+B\in \mathcal{O}_F[x]$ be irreducible over $F$, with $d>s\ge 1$, $(d,s)=1$, let $\Delta=\tu{Disc}(\mathcal{O}_{F(\beta)})$ for $\beta$ a root of $f$, let $\delta=\tu{Disc}(f)$, and let $L$ be the Galois closure of $F(\beta)/F$. Let $\frak{p}$ be a prime of $\mathcal{O}_F$ such that $\frak{p} \nmid AB\mathcal{O}_F$. If $\frak{p}$ ramifies in $\mathcal{O}_{F(\beta)}$, then $e(\frak{q}/\frak{p})=2$ for any prime $\frak{q}$ in $\mathcal{O}_L$ lying above $\frak{p}$. In particular, $\Gal(L/F)$ contains a transposition.\end{thm}
 \begin{proof} Let $\frak{p}$ be a prime of $\mathcal{O}_F$ satisfying the above conditions, and suppose $\frak{p}$ ramifies in $\mathcal{O}_{F(\beta)}$. The discriminant of $f(x)$ is given by \begin{equation}{\label{disctrinomial}}\delta=(-1)^{d(d-1)/2}B^{s-1}(d^dB^{d-s}+(-1)^{d-1}(d-s)^{d-s}s^sA^d).\end{equation} See Lemma 4 of \cite{Smith} for details. Hence if $\frak{p}\mid ds(d-s)\mathcal{O}_F$, then $\frak{p}\nmid (\delta)$ (as $(d,s)=1$ implies $\frak{p}$ cannot divide both $(d)$ and $(s)$ in $\mathcal{O}_F$).  Therefore assume $\frak{p}\nmid ds(d-s)\mathcal{O}_F$. Since $\frak{p}\nmid ds(d-s)A\mathcal{O}_F$, $f'(x)=x^{s-1}(dx^{d-s}+sA)$ has no multiple factors mod $\frak{p}$ other than $x$. It follows that every irreducible factor of $f(x)\Mod{\frak{p}}$ has multiplicity strictly less than three.  Let $\eta$ be a multiple root of $f(x)\Mod{\frak{p}}$ in an algebraic closure of $\mathcal{O}/\frak{p}$. We have \[\eta^{d-s}=-\fr{sA}{d}\in \mathcal{O}/\frak{p}\] and \[\eta^s=-\fr{B}{\eta^{d-s}+A}=-\fr{dB}{(d-s)A}\in \mathcal{O}/\frak{p}.\]  As $(d,s)=1$, we deduce $\eta\in \mathcal{O}/\frak{p}$. If $\xi\in \mathcal{O}/\frak{p}$ is another multiple root of $f\Mod{\frak{p}}$, then from $(\xi/\eta)^s=(\xi/\eta)^{d-s}=1$ and $(s,d-s)=1$, we obtain $\xi=\eta$. Therefore if $\frak{p}\mid (\delta)$, then the factorization of $f(x)$ into irreducible factors mod $\frak{p}$ is: 
 
 \[f(x)=(x-\eta)^2\varphi_1(x)\cdots \varphi_r(x)\Mod{\frak{p}}\] where the $\varphi_i(x)$ are all distinct and separable. Let $\frak{q}$ be any prime above $\frak{p}$ in $\mathcal{O}_L$. Write \[\prod_{i=1}^{r'} (x-\beta_i)\equiv(x-\eta)^2\varphi_1(x)\cdots \varphi_r(x)\Mod{\frak{q}}\] where $\beta_1,\dots,\beta_{r'}$ are the distinct roots of $f$ in $L$, with
 \[(x-\beta_1)(x-\beta_2)\equiv (x-\eta)^2\Mod{\frak{q}}.\] Now choose any $\sigma\in I(\frak{q}/\frak{p})$, $\sigma\ne 1$. Since $(x-\eta)\varphi_1(x)\cdots \varphi_r(x)$ is separable mod $\frak{p}$, we know $\beta_2,\beta_3,\dots,\beta_{r'}$ are mutually non-congruent mod $\frak{q}$. Hence $\sigma(\beta_i)=\beta_i$ for all $i\ge 3$. Since $\sigma\ne 1$, this forces $\sigma(\beta_1)=\beta_2,\sigma(\beta_2)=\beta_1$. Therefore $I(\frak{q}/\frak{p})\cong \Z/2\Z$, and $\Gal(L/F)$ contains a transposition. \end{proof}
  
  We next state several definitions and formulas that will be used regularly in subsequent proofs. For all $n\ge 1$, let $K_n$ denote the splitting field over $K$ of $\varphi^n(x)$, and let $K_0=K$. Let $\mathcal{O}_K$ denote the ring of integers of the number field $K$.
  
\begin{definition*} For $\alpha\in\mathcal{O}_K$ and $\varphi\in\mathcal{O}_K[x]$, we say a prime divisor $\frak{p}$ of $\varphi^n(\alpha)$ is a \textup{primitive prime divisor} (of $\varphi^n(\alpha)$) if, for all $m<n$, $\frak{p}\nmid \varphi^m(\alpha)$. \end{definition*}

\begin{definition*} For $y\in\Z$, the \tu{support of $y$}, denoted $\tu{Supp}(y)$, is the set of positive primes in $\Z$ dividing $y$.\end{definition*}

\begin{definition*} Let $\varphi(x)\in \mathcal{O}_K[x]$ have degree $d$. For $n\ge 2$, we say $H_n=\Gal(K_n/K_{n-1})$ is \textup{maximal} if $H_n=S_d^m$, where $m=deg(\varphi^{n-1}(x))$, and accordingly refer to the extension $K_n/K_{n-1}$ as \tu{maximal}. When $n=1$, $G_1=\Gal(K_1/K)$ is said to be \tu{maximal} if $G_1(\varphi)\cong S_d$.

\end{definition*} It is useful to note the following discriminant formula, whose proof can be found in \cite{Farshid}. Let $\psi\in K[x]$ be of degree $d$ with leading coefficient $\alpha$.

\begin{equation}{\label{discformula}}\tu{Disc}_x(\psi(x)-t)=(-1)^{(d-1)(d-2)/2}d^d\alpha^{d-1}\prod_{b\in R_{\psi}} (t-\psi(b))^{e(b,\psi)} \end{equation} where $R_{\psi}$ denotes the ramification locus of $\psi$, $e(b,\psi)$ denotes the multiplicity of the critical point $b$, and $t$ plays the role of any constant in $K$. From this we obtain \begin{equation}\tu{Disc}_x(\psi^n(x)-t)=(-1)^{(d^n-1)(d^n-2)/2}d^{nd^n}\alpha^{(d^n-1)/(d-1)}\prod_{c\in R_{{\psi}^n}}(t-\psi^n(c))^{e(c,\psi^n)}\end{equation} which is equal to \begin{equation}{\label{disciteratesformula}}(-1)^{(d^n-1)(d^n-2)/2}d^{nd^n}\alpha^{(d^n-1)/(d-1)}\prod_{b\in R_{\psi}}((t-\psi(b))^{d^{n-1}}(t-\psi^2(b))^{d^{n-2}}\cdots(t-\psi^n(b)))^{e(b,\psi)}.\end{equation} Substituting an appropriate value of $t$, we will be using (\ref{discformula}) as an expression for the discriminant of certain trinomials over number fields. Taking $t=0$, the expression in (\ref{disciteratesformula}) shows that for a monic, irreducible $\varphi\in\mathcal{O}_K[x]$, any prime outside $d$ ramifying in $K_n(\varphi)$ must divide $\varphi^m(a)$ for some $m\le n$, where $a$ is some critical point of $\varphi$.

\begin{prop}{\label{prop:LNV}} Let $K$ be a number field, with $\mathcal{O}=\mathcal{O}_K$ its ring of integers. Suppose $\varphi(x)=x^d+Ax^s+B\in \mathcal{O}_K[x]$ has all of its iterates irreducible over $K$, with $(d,s)=1$, $d>s\ge 1$, and all critical points lying in $K$. For a given $n\ge 2$, let $\alpha_1,\dots,\alpha_m$ be the roots of $\varphi^{n-1}(x)$, and let $M_i$ be the splitting field of $\varphi(x)-\alpha_i$ over $K(\alpha_i)$. Suppose that for some critical point $a$ of $\varphi$ with $e(a,\varphi)=1$, there exists a prime divisor $p_n$ of $\varphi^n(a)$ such that $v_{p_n}(\varphi^n(a))$ is odd, and $\tu{gcd}(p_n,dA\varphi^n(b)\varphi^{n-1}(B))=1$ for all critical points $b\ne a$ of $\varphi(x)$. Then for all $i$, $\Gal(M_i/K(\alpha_i))$ contains a transposition.

Furthermore, if $n=1$, the same hypotheses imply that $\Gal(K_1/K)$ contains a transposition.
\end{prop}

\begin{proof} When $n=1$, the statement follows immediately from Theorem \ref{thm:LNV}.  Therefore assume $n\ge 2$. Noting that $N_{K(\alpha_i)/K}(\varphi(a)-\alpha_i)=\varphi^n(a)$, the hypotheses imply that there is a prime $\frak{p}$ of $\mathcal{O}_{K(\alpha_i)}$ such that $v_{\frak{p}}(\varphi(a)-\alpha_i)$ is odd, and $\frak{p}$ lies above $p_n$. Moreover, since $N_{K(\alpha_i)/K}(\varphi(b)-\alpha_i)=\varphi^n(b)$ for all critical points $b$, and since the only primes possibly dividing the denominators of the $\varphi(b)-\alpha_i\in K(\alpha_i)$ are the prime divisors of $d$, it follows from $(\ref{discformula})$ that $v_{\frak{p}}(\varphi(a)-\alpha_i)=v_{\frak{p}}(\tu{Disc}(\varphi(x)-\alpha_i))$. Therefore $\frak{p}$ divides $\tu{Disc}(\varphi(x)-\alpha_i)$ to odd multiplicity.

By Capelli's lemma (see \cite{Capelli} for a statement and proof), the irreducibility of $\varphi^n(x)$ over $K$ implies that $\varphi(x)-\alpha_i$ is irreducible over $K(\alpha_i)$. Hence $\frak{p}$ must ramify in $M_i$. We also observe that $\frak{p}\nmid A(B-\alpha_i)$, as $N_{K(\alpha_i)/K}(B-\alpha_i)=\varphi^{n-1}(B)$ is coprime to $p_n$. Thus, by Theorem \ref{thm:LNV}, it follows that $\Gal(M_i/K(\alpha_i))$ contains a transposition corresponding to $I(\frak{q}/\frak{p})$ for any prime $\frak{q}$ in $\mathcal{O}_{M_i}$ lying above $\frak{p}$. \end{proof} 
 
 We now set forth a criterion ensuring maximality of the extension $K_n/K_{n-1}$. The ideas are inspired by the proof of Theorem 3.1 of \cite{Jamie}.
 
 \begin{prop}{\label{prop:max}} Assume the hypotheses and notation of Proposition \ref{prop:LNV}; suppose, moreover, that $p_n\nmid \varphi^m(a)$ for all $1\le m\le n-1$, and that for all critical points $b\ne a$ of $\varphi(x)$ and $m\le n$, we have $p_n\nmid \varphi^m(b)$. Then $\Gal(M_i/K(\alpha_i))\cong S_d$ implies $\Gal(K_n/K_{n-1})$ is maximal. \end{prop} 
 \begin{proof} Let $\widehat{M}_i=K_{n-1}\prod_{j\ne i} M_j$. Throughout the proof, it will be convenient to refer to the following diagram of field extensions:
 
 
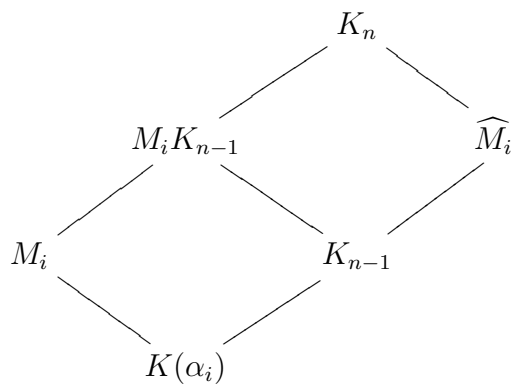
\[\centerline{
\xymatrix{
    &   &  K_n\makebox[0pt][l]{${}$} \ar@{-}[dl] \ar@{-}[dr] &  \\
    & M_iK_{n-1}\ar@{-}[dl]\ar@{-}[rd] &  & \widehat{M}_i \makebox[15pt][l]{${}$} \ar@{-}[dl]\\
  M_i\ar@{-}[rd] &  & K_{n-1}\makebox[0pt][l]{${}$} \ar@{-}[dl] &  &\\
   & K(\alpha_i) &  & \makebox[0pt][l]{${}$}
}}\hspace{6em}
\]

\vspace{5mm} We first argue that $\Gal(M_iK_{n-1}/K_{n-1})\cong S_d$. Note that $K_{n-1}/K(\alpha_i)$ is a normal extension, so base changing by $K_{n-1}$, we obtain $\Gal(M_iK_{n-1}/K_{n-1})$ as a normal subgroup of $\Gal(M_i/K(\alpha_i))\cong S_d$. Let $\frak{p}$ and $\frak{q}$ be as in the proof of Proposition \ref{prop:LNV}, let $\frak{Q}$ be any prime of $M_iK_{n-1}$ lying above $\frak{q}$, and let $\frak{P}=\frak{Q}\cap K_{n-1}$.  Then $e(\frak{P}/\frak{p})=1$; this is because $p_n$ is a primitive prime divisor of $\varphi^n(a)$, so (\ref{disciteratesformula}) implies $p_n=\frak{p}\cap K$ doesn't ramify in $K_{n-1}$. Thus $e(\frak{Q}/\frak{q})=1$ as well. By the proof of Proposition \ref{prop:LNV}, $e(\frak{q}/\frak{p})=2$, so this forces $e(\frak{Q}/\frak{P})=2$. Hence $\Gal(M_iK_{n-1}/K_{n-1})$ is a normal subgroup of $S_d$ containing a transposition, meaning $\Gal(M_iK_{n-1}/K_{n-1})\cong S_d$.

Let $\frak{Q'}$ be a prime of $K_n$ lying above $\frak{Q}$, and let $\frak{P}'$ be the prime in $\widehat{M}_i$ lying below $\frak{Q}'$.

We know that $\frak{P}$ cannot divide $\tu{Disc}(\varphi(x)-\alpha_j)$ for $j\ne i$; indeed, otherwise, $\frak{P}$ divides either $\varphi(a)-\alpha_j$, or $\varphi(b)-\alpha_j$ for some critical point $b\ne a$. In the former case, $\frak{P}$ then divides $\alpha_i-\alpha_j\mid \tu{Disc}(\varphi^n)$, contradicting the fact that $p_n=\frak{P}\cap K$ does not ramify in $K_{n-1}$. In the latter case, $\frak{P}$ divides $N_{K(\alpha_i)/K}(\varphi(b)-\alpha_i)=\varphi^n(b)$, so $p_n\mid \varphi^n(b)$, contradicting our hypotheses.  Hence $e(\frak{P}'/\frak{P})=1$. 

This forces $e(\frak{Q}'/\frak{Q})=1$. We then see that $e(\frak{Q}'/\frak{P})=2$, and thus $e(\frak{Q}'/\frak{P}')=2$. In particular, $\Gal(K_n/\widehat{M}_i)$ contains a transposition. As it is a normal subgroup of $\Gal(M_iK_{n-1}/K_{n-1})\cong S_d$, this implies $\Gal(K_n/\widehat{M}_i)\cong S_d$, and so $\Gal(K_n/K_{n-1})\cong S_d^m$, $m:=deg((\varphi^{n-1})$. \end{proof}
 
\section{Odoni's Conjecture in Prime Degree}

Let $\phi_{(p,k)}(x)=x^p+kpx^{p-1}-kp$, with $p\ge 3$ prime, $k\in\Z^+$, and $p\nmid k$. The finite critical points of $\phi_{(p,k)}$ are 0 and $a_{(p,k)}=-(p-1)k$, with multiplicities $p-2$ and 1, respectively. Since there is no risk of confusion, we drop the dependence of the nonzero critical point on $(p,k)$ and denote it as $a$. The point 0 is strictly preperiodic under $\phi_{(p,k)}$, as $\phi_{(p,k)}(0)=-kp$, and $\phi_{(p,k)}(-kp)=-kp$. 

\begin{lem}{\label{lem:PPD}} Let $\varphi(x)$ be in $\Z[x]$, and let $b\in \Z$. If a prime $q$ dividing $\varphi^n(b)$ is not a primitive prime divisor of $\varphi^n(b)$, then \begin{equation}{\label{refinement}}q\mid \varphi^m(b) \tu{ or } q\mid \varphi^m(0)\tu{ for some } 1\le m\le \Bigl\lfloor\frac{n}{2}\Bigr\rfloor.\end{equation} \end{lem}

\begin{proof} Let $q$ be a prime dividing $\varphi^n(b)$, and suppose $q\mid \varphi^m(b)$ for some $m<n$. Then $0=\varphi^n(b)=\varphi^{n-m}(\varphi^m(b))\Mod{q}=\varphi^{n-m}(0)\Mod{q}$. \end{proof}

\begin{rmk} By the proof of Lemma \ref{lem:PPD}, we see that any prime divisor of $\varphi^n(b)$ that does not divide any $\varphi(0),\varphi^2(0),\dots,\varphi^{n-1}(0)$ is a primitive prime divisor of $\varphi^n(b)$. In particular, if $q$ is a prime divisor of $\varphi^n(b)$ such that $q$ does not divide any element of the forward orbit of 0 under $\varphi$, then $q$ is a primitive prime divisor. 

Thus, for the maps $\phi_{(p,k)}$, the prime divisors of $\phi_{(p,k)}^n(a)$ that are coprime to $kp$ are primitive prime divisors. \end{rmk}

\begin{lem}{\label{lem:sign}} Let $\phi_{(d,k)}(x)=x^d+kdx^{d-1}-kd$, where $d\ge 3$ is odd, and $k\in\Z^+$. Let $a=-k(d-1)$. Then $\phi_{(d,k)}^n(a)>0$ for all $n\ge 1$. \end{lem}

\begin{proof} We have $\phi_{(d,k)}(y)=y^{d-1}(y+kd)-kd$, so if $|y|<kd$ and $y^{d-1}>kd$ for $y\in\Z$, then $\phi_{(d,k)}(y)>0$. Then $\phi_{(d,k)}^2(y)>0$, and by induction $\phi_{(d,k)}^n(y)>0$ for all $n\ge 1$. Observing that $a$ satisfies $|a|<kd$ and $a^{d-1}>kd$, this completes the proof. \end{proof}

The above lemmas allow us to derive Proposition \ref{prop:trimax}, a key tool in the proof of Theorem \ref{thm:primeOdoni}. It relies heavily on the fact 0 is strictly preperiodic under $\phi_{(p,k)}$; compare Theorem 1.1 of \cite{Jonesdensity}, as well as Theorem 1.2 of \cite{Jones2}.

\begin{prop}{\label{prop:trimax}} Let $\phi_{(p,k)}(x)=x^p+kpx^{p-1}-kp$, where $k\in\Z^+$, $p\ge 3$ is prime, and $p\nmid k$. Fix some $n\ge 1$. If $\phi_{(p,k)}^n(a)\ne ky^2$ for any $y\in\Z$, then $\Gal(K_n/K_{n-1})$ is maximal. \end{prop}

\begin{proof} First note that $\phi_{(p,k)}(x)$ and all of its iterates are Eisenstein at $p$, and hence irreducible over $\Q$. By Capelli's lemma, it follows that for any $n\ge 2$, $\phi_{(p,k)}(x)-\alpha_i$ is irreducible over $\Q(\alpha_i)$ for any root $\alpha_i$ of $\phi_{(p,k)}^{n-1}$. Furthermore, if $q\mid k$ is prime, then $v_q(\phi_{(p,k)}^n(a))=v_q(k)$ for all $n\ge 1$. We also have $(\phi_{(p,k)}^n(a),p)=1$. By these facts, and by Lemma \ref{lem:sign}, having $\phi_{(p,k)}^n(a)\ne ky^2$ for any $y\in\Z$ implies that some prime divisor $p_n$ outside $\tu{Supp}(kp)$ divides $\phi_{(p,k)}^n(a)$ to odd multiplicity; this $p_n$ must be a primitive prime divisor. Proposition \ref{prop:LNV} then implies that $\Gal(M_i/\Q(\alpha_i))$ contains a transposition. A theorem of Jordan \cite{Jordan} states that if $H$ is a transitive subgroup of $S_p$, where $p$ is prime, which contains a transposition, then $H=S_p$. As $\phi_{(p,k)}(x)-\alpha_i$ is irreducible over $\Q(\alpha_i)$, $\Gal(M_i/\Q(\alpha_i))$ is a transitive subgroup of $S_d$. Hence $\Gal(M_i/\Q(\alpha_i))\cong S_p$.
Moreover, this prime $p_n$ dividing $\phi_{(p,k)}^n(a)$ to odd multiplicity clearly does not divide any $\phi_{(p,k)}^m(0)$ with $m\le n$, as the forward orbit of 0 under $\phi_{(p,k)}$ is $\{-kp\}$. Therefore, by Proposition \ref{prop:max}, $K_n/K_{n-1}$ is maximal for all $n\ge 2$.

Now let $n=1$. In the notation of Proposition \ref{prop:LNV}, $d=p$, $A=kp$, $b=0$, $\varphi^0(B)=B=-kp$. Thus $\phi_{(p,k)}(a)\ne ky^2$ for any $y\in\Z$ implies that $G_1(\phi_{(p,k)})$ is maximal. \end{proof}

\begin{rmk} If the hypothesis of Proposition \ref{prop:trimax} is satisfied for infinitely many $n$, then it follows automatically from the proof of the proposition that $a$ has infinite forward orbit under $\phi_{(p,k)}(x)$, so we do not need to take this property into account. \end{rmk}

We now prove Theorem \ref{thm:primeOdoni}.

\begin{proof}[Proof of Theorem \ref{thm:primeOdoni}] Suppose $p\ge 5$, and $k\equiv 1\Mod{3}$. The finite critical points of $\phi_p$ are 0 and $a=-(p-1)k$, of multiplicities $p-2$ and 1 respectively. Clearly, the forward orbit of 0 under $\phi_{(p,k)}$ is $\{-kp\}$. Since $(p,k)=1$, showing that $K_n/K_{n-1}$ is maximal amounts to showing that $\phi_{(p,k)}^n(a)\ne ky^2$ in $\Q$. Note also that \[p\equiv 1\Mod{3}\llra a\equiv 0\Mod{3}\] and \[p\equiv -1\Mod{3}\llra a\equiv -1\Mod{3}.\]
In both cases, it is easy to check that $\phi_{(p,k)}(a)\equiv-1\Mod{3}$. Since $-1$ is a fixed point modulo 3 in both cases, this implies that for all $n\ge 1$, $\phi_{(p,k)}^n(a)$ is not of the form $ky^2$ in $\Z$, $k\equiv 1\Mod{3}$. Hence $\Gal(K_n/K_{n-1})$ is maximal for all $n\ge 1$. 

Now let $p=3$. Then taking $k=2$ yields the map $\phi_{(3,2)}(x)=x^3+6x^2-6$. By Proposition \ref{prop:trimax}, for $n\ge 1$, $K_n/K_{n-1}$ is maximal as long as $\phi_{(3,2)}^n(a)\ne 2z^2$ for $z\in\Z$. Multiplying both sides by 2, and substituting $y=2z$, we see that it suffices to show that the elliptic curve with affine equation \[\mathcal{C}_1: y^2=2(x^3+6x^2-6)\] does not have integral points with $x$-coordinate equal to $\phi_{(3,2)}^n(a)$ for some $n$. This curve is isomorphic over $\Q$ to \[\mathcal{C}_2: y^2=x^3+48x^2-3072\] via the change of coordinates taking a solution $(x,y)$ of $\mathcal{C}_1$ to $(8x,16y)$ on $\mathcal{C}_2$. Thus, to find integral points on $\mathcal{C}_1$, it is enough to find integral points on $\mathcal{C}_2$. But a \texttt{\texttt{\texttt{Magma}}} \cite{Magma} search reveals that there are no nontrivial integral points on the projectivization of $\mathcal{C}_2$. Therefore $K_n/K_{n-1}$ is maximal for all $n$, so $G_{\Q}(\phi_{3,2})=\tu{Aut}(T_{\infty})$. This completes the proof of Odoni's conjecture in all prime degrees. \end{proof}

\begin{proof}[Proof of Corollary \ref{cor:density}] As already noted, $\phi_{(p,k)}^n(x)$ is Eisenstein at $p$ for all $n\ge 1$; hence the iterates of $\phi_{(p,k)}$ are all irreducible over $\Q$. Moreover, if $a_0$ has infinite forward orbit under $\phi_{(p,k)}$, then for all $n$ we have $\phi_{(p,k)}^n(a_0)\ne 0$. Theorem 4.1 of \cite{Jones} immediately completes the proof. \end{proof}

Proposition \ref{prop:trimax} allows us to easily deduce Theorem \ref{thm:primefinindex}.

 \begin{proof}[Proof of Theorem \ref{thm:primefinindex}] By Proposition \ref{prop:trimax}, having $\phi_{(p,k)}^n(a)\ne ky^2$ for $y\in \Z$ implies that $\Gal(K_n/K_{n-1})\cong S_p^m$. By Siegel's theorem, since the orbit of $a$ under $\phi_{(p,k)}(x)$ is infinite, there are finitely many $n$ such that $\phi^n(a)$ is of this form. \end{proof}
 
Next, we use the proof of Proposition \ref{prop:trimax} to prove Theorem \ref{thm:index2}.
 
 \begin{proof}[Proof of Theorem \ref{thm:index2}] The critical point of $\phi$ is $a=-14/3$.  It is easy to see that $v_7(\phi^n(a))=1$ for all $n$, and that $v_3(\phi^n(a))$ is odd for all $n$. Since $v_3(\phi^n(a))\to-\infty$ as $n\to\infty$, it follows that $a$ has infinite forward orbit under $\phi$. One checks that $\phi^2(a)>1$, and thus by induction $\phi^n(a)>1>0$ for all $n\ge 1$. Moreover, $\phi^n(a)$ is an $S$-integer for all $n$, where $S=\{3,\infty\}$.
 
Since $\phi^n(0)=-7$ for all $n\ge 1$, we deduce as in the proof of Lemma \ref{lem:PPD} that any prime divisor of the numerator of $\phi^n(a)$ that is coprime to $7$ does not divide the numerator of any $\phi^m(a)$, $m\le n-1$. Therefore, by the proof of Proposition \ref{prop:trimax}, $K_n/K_{n-1}$ is maximal for any $n\ge 1$ as long as $\phi^n(a)\ne 7\cdot3y^2=21y^2$ for any $y\in\Z_S$. A computation in \texttt{\texttt{Magma}} shows that the elliptic curves with affine models \[\mathcal{C}_3: y^2=21(x^3+7x^2-7)\] and \[\mathcal{C}_4: y^2 = x^3 + 583443x^2 - 4053211077702843\] are isomorphic over $\Q$, via the coordinate change taking a solution $(x,y)$ on $\mathcal{C}_3$ to $(3^57^3x,3^77^4y)$ on $\mathcal{C}_4$. The set of $S$-integral points on $\mathcal{C}_3$ is thus contained in the set of $S$-integral points of $\mathcal{C}_4$. There are exactly three pairs of $S$-integral points of $\mathcal{C}_4$; their inverse images in $\mathcal{C}_3$ are: \[(-206/189, \pm 377/567)\] \[(7/3, \pm 91/3)\] \[(-14/3,\pm 91/3).\] The points $(-14/3, \pm 91/3)$ correspond to $21\cdot \phi(a)=(\pm\frac{91}{3})^2$; indeed, $G_1$ is not maximal, and $G_1(\phi)\cong \Z/3\Z$, an index 2 subgroup of $S_3$. Clearly, $\phi^n(a)$ never equals $7/3$, and $-206/189$ is not an $S$-integer. Therefore $\Gal(K_n/K_{n-1})$ is maximal for all $n\ge 2$, and we conclude that $G_{\Q}(\phi)$ has index 2 in $\tu{Aut}(T_{\infty})$. \end{proof}

 \section{Double Transitivity}
 
In this section, we generalize a specific case addressed by Theorem 1.3 in the article \cite{Salinier1} by showing that it holds in relative extensions of number fields. 

Let $F$ be a number field, with $\mathcal{O}_F$ its ring of integers. For a prime $\frak{p}$ of $\mathcal{O}_F$, let $v_{\frak{p}}$ be the normalized valuation corresponding to $\frak{p}$, and let $F_{\frak{p}}$ be the completion of $F$ with respect to $v_{\frak{p}}$. Let $f(x)=x^d+Ax^{d-1}+B\in O_F[X]$ be irreducible over $F$, with $d\ge 3$. Let $L$ be the splitting field of $f(x)$ over $F$. For a prime $\frak{p}$ in $\mathcal{O}_F$, we fix a prime $\frak{P}$ of $\mathcal{O}_L$ above $\frak{p}$ and let $L_{\frak{P}}$ be the completion of $L$ with respect to the $\frak{P}$-adic valuation. 

The main result of this section is the following.

\begin{thm}{\label{thm:2trans}} Suppose there exists a prime $\frak{p}$ of $\mathcal{O}_F$ dividing $(B)$ but not $(A)$ such that $(d-1,v_{\frak{p}}(B))=1$. Then $\Gal(L/F)$ contains a subgroup $H$ acting transitively on $d-1$ roots of $f(x)$ and fixing the remaining root.  Hence, $\Gal(L/F)$ is a doubly transitive subgroup of $S_d$. \end{thm}

To prove Theorem \ref{thm:2trans}, we first prove a lemma. Let $\frak{p}$ be a prime divisor of $(B)$ but not $(A)$. By the general form of Hensel's lemma (see p.129 of \cite{Neukirch}), we have 

\[f(x)=g(x)h(x)\] over the valuation ring of $F_{\frak{p}}$, where \[g(x)\equiv x^{d-1}\Mod{\frak{p}} \tu{ and } h(x)\equiv x+A\Mod{\frak{p}}\] with $\tu{deg}(g(x))=d-1$, $\tu{deg}(h(x))=1$ over $F_{\frak{p}}$. 

\begin{lem}{\label{lem:Sal3}} Let $\frak{p}$ be as above, and suppose that $(d-1,v_{\frak{p}}(B))=1$. Then, for each root $\beta$ of $g(x)$, the extension $F_{\frak{p}}(\beta)/F_{\frak{p}}$ is totally ramified, and $g(x)$ is irreducible over $F_{\frak{p}}$. \end{lem}

\begin{proof} Let $w$ be the normalized valuation of the local field $F_{\frak{p}}(\beta)$. Then $w(\frak{p})=e$, where $e$ is the ramification index of $F_{\frak{p}}(\beta)/F_{\frak{p}}$. Since $g(x)\equiv x^{d-1}\Mod{\frak{p}}$, we have $w(\beta)>0$, and since $f(\beta)=\beta^d+A\beta^{d-1}+B=0$, \[(d-1)w(\beta)=w(B)=ev_{\frak{p}}(B)\] By hypothesis, $(d-1,v_{\frak{p}}(B))=1$, so this forces $d-1$ to divide $e$. On the other hand, we see that \[e\le [F_{\frak{p}}(\beta):F_{\frak{p}}]\le d-1=\tu{deg}(g(x))\] meaning $F_{\frak{p}}(\beta)/F_{\frak{p}}$ is a totally ramified extension of $F_{\frak{p}}$, and $g(x)$ is irreducible over $F_{\frak{p}}$. \end{proof}

We now prove Theorem \ref{thm:2trans}.

\begin{proof}[Proof of Theorem \ref{thm:2trans}] By the previous lemma, $g(x)$ is irreducible over $F_{\frak{p}}$. Thus $\Gal(L_{\frak{p}}/F_{\frak{p}})$ is transitive on the roots of $g(x)$. Clearly, $\Gal(L_{\frak{p}}/F_{\frak{p}})$ is not supported on the root of $h(x)$, as $h(x)$ is linear with its root lying in $F_{\frak{p}}$. Viewing $\Gal(L/F)$ as a group acting on the set $X$ of roots of $f$, this proves the existence of an $x\in X$ such that the stabilizer of $x$ acts transitively on the remaining $d-1$ elements of $X$. The irreducibility of $f$ then implies that this property holds for any $x\in X$. Since $|X|\ge 3$, this is equivalent to $\Gal(L/F)$ being a doubly transitive subgroup of $S_d$ \cite{Rotman}. \end{proof}

\section{Application of Vojta's conjecture}

Throughout this section, we work with the family of polynomials $\phi_{(d,c)}(x)=x^d-cdx^{d-1}+c(d-1)\in\Z[x]$ with $c\in\Z^+$, $d\ge 3$. We often abbreviate $\phi_{(d,c)}(x)$ as $\phi_c(x)$. Let $H:\overline{\Q}\to\mathbb{R}_{\ge 1}$ denote the absolute multiplicative height, and $h:\overline{\Q}\to\mathbb{R}_{\ge 0}$ the absolute logarithmic height of an algebraic number (see \cite{Bombieri} for definitions). Although the base field is $\Q$ throughout, we use the notation of heights, as the ideas in these proofs can be generalized to arbitrary number fields.

\begin{lem}{\label{lem:orb}} Let $\phi_{(d,c)}(x)=x^d-cdx^{d-1}+c(d-1)$ be as above. Then the forward orbit of 0 under $\phi_{(d,c)}$ is infinite. \end{lem}

\begin{proof} For all $n\ge 1$, $\phi_{(d,c)}^n(0)=c(d-1)\cdot y$, for some $y\in\Z$. It is easily checked that \[\phi_c^{n+1}(0)=c(d-1)[c^{d-1}(d-1)^{d-2}y^{d-1}(y(d-1)-d)+1].\]  Thus when $H(y)\ge 2$, we have \[H(\phi_c^{n+1}(0))\ge c^d(d-1)^{d-1}H(y^{d-1})=c\cdot H(\phi_c^n(0))^{d-1}.\] Noting that $\phi_{(d,c)}^2(0)=c(d-1)[1-c^{d-1}(d-1)^{d-2}]$, it is clear that $H(\phi_{(d,c)}^2(0))\le c(d-1)$ only when $c=1$, $d=3$. One computes that \[\phi_{(3,1)}^3(0)=-18=c(d-1)y\] for $H(y)=9>2$. Therefore, for all $n\ge 3$, \begin{equation} H(\phi_c^{n+1}(0))\ge c\cdot H(\phi_c^n(0))^{d-1}\ge H(\phi_c^n(0))^{d-1}\end{equation} which yields \begin{equation}{\label{lowerbound}} H(\phi_c^{n+1}(0))=H(\phi_c^n(c(d-1)))\ge (c(d-1))^{(d-1)^{n-3}}.\end{equation}

In particular, the forward orbit of 0 under $\phi_{(d,c)}$ is infinite. \end{proof}

For notational clarity, we now denote by $C$ an indeterminate, and by $c$ a specific value of $C$.

\begin{lem}{\label{lem:simpleroots}} Fix $d\ge 3$. The polynomial $\phi_{(C,n)}(0):=\phi_{(C,d)}^n(x)\mid_{x=0}\in \Z[C]$ has simple roots in $\overline{\Q}$ for all $n\ge 1$. \end{lem}

\begin{proof} We first argue by induction that for all $n\ge 1$, the leading coefficient of $\phi_{(C,n)}(0)$ is $\pm(d-1)^m$ for some $m\ge 1$. When $n=1$, this is clear. Suppose $n\ge 2$, and that $\phi_{(C,n)}(0)$ has leading term $\pm(d-1)^jC^{d^{n-1}}$ for some $j\ge 1$. Then $\pm (d-1)^{jd}$ is the leading coefficient of $(\phi_{(C,n)}(0))^d$. On the other hand, $Cd(\phi_{(C,n)}(0))^{d-1}$ has degree $d^{n-1}(d-1)+1$, which is strictly less than $d^n$, as $n\ge 2$. Hence the leading coefficient of $\phi_{(C,n+1)}(0)$ is $\pm (d-1)^{jd}$. On the other hand, one computes that $\phi_{(C,2)}(0)$ has leading term $-(d-1)^{d-1}c^d$. Therefore by induction, for all $n$, the leading coefficient of $\phi_{(C,n)}(0)$ is $\pm(d-1)^m$, with $m\ge 1$. 

For a finite set $S$ of primes of $\Z$ including the infinite place, $\alpha\in\overline{\Q}$, and $K:=\Q(\alpha)$, we write $\alpha\in\overline{\Z}_S$ if $v_{\frak{p}}(\alpha)\ge 0$ for all primes $\frak{p}$ in $\mathcal{O}_K$ such that $\frak{p}\notin S'$, where $S'$ is the set of places of $K$ lying above $S$. If we let $S$ be the set of primes of $\Z$ dividing $d-1$ along with the infinite place, then any root $c$ of $\phi_{(C,n)}(0)$ is an element of $\overline{\Z}_S$. Thus, if $c$ is a multiple root of $\phi_{(C,n)}(0)$, where $n\ge 2$, then \[0\equiv (\phi_{(C,n)}(0))'\mid_{C=c}\equiv[(\phi_{(C,n-1)}(0))^d-Cd(\phi_{(C,n-1)}(0))^{d-1}+C(d-1)]'\mid_{C=c}\equiv -1\Mod{d\overline{\Z}_S},\] and we reach a contradiction. \end{proof}

For a fixed $d\ge 3$, define $\Phi_{(C,n)}(x)=\prod_{m\mid n}(\phi_{(d,C)}^{m}(x)-x)^{\mu(n/m)}\in \Z[C,x]$, and $\Phi_{(c,n)}(x)=\prod_{m\mid n}(\phi_{(d,c)}^m(x)-x)^{\mu(n/m)}\in\Z[x]$. 

\begin{lem}{\label{lem:dynatomic}} The polynomial $\Phi_{(C,n)}(0)\in\Z[C]$ has at least three simple roots in $\overline{\Q}$ for all $n\ge 3$. \end{lem}

\begin{proof} Each $\phi_{C,n}(0)\in\Z[C]$ has degree $d^{n-1}$.  By Lemma \ref{lem:simpleroots}, we see that in order to guarantee that $\Phi_{C,n}(0)$ has at least three simple roots, it suffices for the total number of roots of $\phi_{C,n}(0)$ to be at least three more than the sum of the number of roots of $\phi_{C,m}(0)$ for all proper divisors $m$ of $n$. 

Let $a_1,a_2,\dots a_r$ be the proper divisors of $n$, with $a_i<a_{i+1}$ for all $i$. Then, since $d\ge 3$, we have \[d^{a_1-1}+d^{a_2-1}+\dots+d^{a_r-1}<d^{a_r}\le d^{n-2}<d^{n-1}-3\] for all $n\ge 3$.\end{proof}

As previously, let $K$ be a number field, with $\mathcal{O}_K$ its ring of integers. Let $\alpha\in \mathcal{O}_K$, and $\varphi(x)\in\mathcal{O}_K[x]$.

\begin{definition*} Let $(\varphi^n(\alpha))=\frak{p}_1^{e_1}\cdots\frak{p}_r^{e_r}\subset \mathcal{O}_K$, where the $\frak{p}_i$ are prime ideals in $\mathcal{O}_K$.  If $n\ge 2$, then the \tu{primitive part} of $\varphi^n(\alpha)$ is the subproduct $\frak{p}_{i_1}^{e_{i_1}}\cdots \frak{p}_{i_s}^{e_{i_s}}\supset \frak{p}_1^{e_1}\cdots\frak{p}_r^{e_r}$, where the $\frak{p}_{i_j}$ do not divide any ideal in the set $\{(\varphi^m(\alpha))\}_{m=1}^{n-1}$, and $v_{\frak{p}_{i_j}}(\varphi^n(\alpha))=e_{i_j}$. \end{definition*}

\begin{definition*} A \tu{rigid divisibility sequence} is a sequence $b_n$ of elements of $K$ such that both of the following conditions hold:
\begin{enumerate}[topsep=5pt, partopsep=5pt, itemsep=2pt]
\item[\textup{(1)}] for all primes $\frak{p}$ of $\mathcal{O}_K$, $v_{\frak{p}}(b_n)>0$ implies $v_{\frak{p}}(b_{mn})=v_{\frak{p}}(b_n)$ for all $m\ge 1$
\item[\textup{(2)}] $v_{\frak{p}}(b_n),v_{\frak{p}}(b_m)\ge e$ implies $v_{\frak{p}}(b_{\tu{gcd}(m,n)})\ge e$.
\end{enumerate} \end{definition*} Conditions (1) and (2) together in fact imply that if $v_{\frak{p}}(b_n)=e>0$ for some $n$, then for any $m\ne n$, either $v_{\frak{p}}(b_m)=0$ or $v_{\frak{p}}(b_m)=e$.

\begin{lem}{\label{lem:primpart}} Suppose $\varphi(0),\varphi^2(0),\dots$ forms a rigid divisibility sequence in $\mathcal{O}_K$. Then the ideal generated by $\prod_{m\mid n}(\varphi^m(0))^{\mu(n/m)}$ in $\mathcal{O}_K$ equals the primitive part of $\varphi^n(0)$. \end{lem}

\begin{proof} Fix $n>1$, and let $\frak{p}$ be any prime ideal dividing one or more elements of $\{(\varphi^m(0))\}_{m=1}^n$. Let $l(\frak{p})$ be the least $m$ such that $\frak{p}\mid (\varphi^m(0))$, and let $e(\frak{p})=v_{\frak{p}}(\varphi^m(0))$. Define $a_m:=\prod_{l(\frak{p})=m} \frak{p}^{e(\frak{p})}$, which is the primitive part of $\varphi^m(0)$.  Since the $a_n$ are coprime, we have $(\varphi^n(0))=\prod_{m\mid n} a_m$. By M{\"o}bius inversion, this gives $a_n=\prod_{m\mid n} (\varphi^m(0))^{\mu(n/m)}$. \end{proof}

We finally prove Theorem \ref{thm:abcOdoni}.

\begin{proof}[Proof of Theorem \ref{thm:abcOdoni}] Fix a $d$ satisfying the hypothesis of the theorem, and let $c\in  \mathcal{B}_d$. The finite critical points of $\phi_{(d,c)}$ are $a:=c(d-1)$ and 0, of multiplicities 1 and $d-2$ respectively. Moreover, $\phi_{(d,c)}(0)=a$. By Lemma \ref{lem:orb}, we know that the orbit of $0$ (and thus the orbit of $a$) under $\phi_{(d,c)}$ is infinite. Since $\phi_{(d,c)}(x)$ has no linear term, it is easy to check that the orbit of 0 under $\phi_{(d,c)}(x)$ forms a rigid divisibility sequence (compare Lemma 5.3 of \cite{Jonesdensity}).

Let $n\ge 2$. Since $\phi_c^n(0)=\phi_c^{n-1}(a)$, any primitive prime divisor $p_n$ of $\phi_c^n(a)$ does not occur as a divisor of any $\phi_c^m(0)$, $m\le n$. We also have $c\mid \phi_c^n(a)$ for all $n$, so any prime divisor of $c$ cannot be a primitive prime divisor of $\phi_c^n(a)$ for any $n\ge 2$. By definition, $\phi_c^{n-1}(c(d-1))=\phi_c^{n-1}(a)$ must be coprime to any primitive prime divisor of $\phi_c^n(a)$. Hence by Proposition \ref{prop:LNV}, to show that $\Gal(M_i/\Q(\alpha_i))$ contains a transposition for any $i$ corresponding to a root $\alpha_i$ of $\phi_c^{n-1}(x)$, it suffices to prove the existence of a primitive prime divisor $p_n$ of $\phi_c^n(a)$, such that $p_n\nmid d$, and $p_n$ divides $\phi_c^n(a)$ to odd multiplicity. (Here we are using the same notation as in Section \ref{section:ram}; for a fixed $d$, we sometimes write $K_n(\phi_c)$ and $\widehat{M}_i(\phi_c)$ to emphasize the dependence of the various $K_n$ on the choice of $c$.)

We first use the Vojta (or Hall-Lang) conjecture to bound the index $n$ such that $K_n/K_{n-1}$ can fail to be maximal, in such a way that this bound is uniform in $c$ (see \cite{Stoll} for a statement of this conjecture for $d\ge 5$, and ~\cite[IX.7]{Silverman} for $d\le 4$). 

For $f\in \overline{\Q}[x]$, let $h(f)$ be the maximum of the logarithmic heights of the coefficients of $f$.

\begin{conj}\label{conj:Vojta} For each $d\ge 3$, there exist constants $C_1=C_1(d)$ and $C_2=C_2(d)$ such that for all $f\in\Z[x]$ of degree $d$ with $\tu{disc}(f)\ne 0$, then if $x,y\in\Z$ satisfy $y^2=f(x)$, then \[h(x)\le C_1(d)\cdot h(f)+C_2(d).\]  \end{conj} \vspace{5mm} Fix $n\ge 3$, and let $K_n/\widehat{M}_i$ correspond to some root $\alpha_i$ of $\phi_c^{n-1}(x)$. From the inequality in (\ref{lowerbound}), we have $H(\phi_c^k(a))\ge a^{{(d-1)}^{k-3}}$ for all $k\ge 3$, so \[h(\phi_c^k(a))\ge (d-1)^{k-3}h(a).\] On the other hand, if $K_n/\widehat{M}_i$ does not contain a transposition, i.e., if $\phi_c^n(a)$ does not have an odd multiplicity primitive prime divisor outside $d$, then by the observation made in $(\ref{refinement})$, $\phi_c(\phi_c^{n-1}(a))=d_ny^2$, where \begin{equation}{\label{supp}}\tu{Supp}(d_n)\subset \tu{Supp}(d\cdot\phi_c(a)\cdots\phi_c^{\lfloor{\frac{n}{2}}\rfloor}(a)).\end{equation} Thus $\phi_c^{n-1}(a)$ is the $x$-coordinate of an integral point on the curve \begin{equation}{\label{diagcurve1}} d_ny^2=\phi_c(x)\end{equation} or, equivalently, on the curve \begin{equation}{\label{diagcurve2}} y^2=d_n\phi_c(x).\end{equation} As $h(d_n\phi_c(x))>0$ for all $c\in\mathcal{B}_d$, taking $f(x)=d_n\phi_c(x)$ in Conjecture \ref{conj:Vojta} gives \begin{equation}{\label{Vojtabound1}} h(\phi_c^{n-1}(a))\le C_d\cdot (h(d_n)+h(\phi_c(x))\end{equation} for some constant $C_d>0$ that is independent of the choice of $c\in\mathcal{B}_d$.

From $(\ref{supp})$, we have $h(d_n)\le h(d)+h(\phi_c(a))\cdots+h(\phi_c^{\lfloor{\frac{n}{2}}\rfloor}(a))$. Moreover, it is easy to see that $h(\phi_c^m(a))\le d^m\cdot h(a)$, which yields \begin{equation}{\label{upperbound}}h(\phi_c(a))+h(\phi_c^2(a))+\cdots+h(\phi_c^{\lfloor{\frac{n}{2}}\rfloor}(a))\le dh(a)+d^2h(a)+\cdots+d^{\lfloor{\frac{n}{2}}\rfloor}h(a)<h(a)d^{\lfloor{\frac{n}{2}}\rfloor+1}.\end{equation} 
This inequality, along with $(\ref{Vojtabound1})$, implies \begin{equation}{\label{Vojtabound2}}h(\phi_c^{n-1}(a))< C_d\cdot (h(d)+d^{\lfloor{\frac{n}{2}}\rfloor+1}h(a)+h(cd)).\end{equation} 
Combining (\ref{lowerbound}) and (\ref{Vojtabound2}), and simplifying, we arrive at \[(d-1)^{n-4}< C_d\cdot(h(d)+d^{\lfloor{\frac{n}{2}}\rfloor+1}+2).\] This relation is clearly violated for large enough $n$.  In addition, the bound $N$ on $n$ does not depend on the choice of parameter $c\in  \mathcal{B}_d$. 

Therefore when $n>N$, we know that for each $i$ corresponding to a root $\alpha_i$ of $\phi_c^{n-1}$, the extension $M_i/\Q(\alpha_i)$ contains a transposition by Proposition \ref{prop:LNV}. In the case where $d$ is prime, we conclude that $M_i/\Q(\alpha_i)\cong S_d$, and so Proposition \ref{prop:max} implies that $K_n/K_{n-1}$ is maximal for all $n>N$. This proves that $G_{\Q}(\phi_{(d,c)})\le \tu{Aut}(T_{\infty})$ has finite index in that case. 

Now let $2\le n\le N$. By Lemma \ref{lem:primpart}, $\phi_c^n(a)=\phi_c^{n+1}(0)$ has an odd multiplicity primitive prime divisor outside $d$ if and only if $\Phi_{c,n+1}(0)\ne d_iy^2$ in $\Z$, for any $d_i\mid d$.  By Lemma \ref{lem:dynatomic}, we know that $\Phi_{C,n}(0)$ has at least three simple roots when $n\ge 3$. Siegel's theorem then implies that there are finitely many $c\in\Z$ such that $\Phi_{c,n}(a)=\Phi_{c,n+1}(0)$ is of the above form. Thus, for all sufficiently large $c\in  \mathcal{B}_d$, Proposition \ref{prop:LNV} implies that for each $2\le n\le N$, the associated extensions $M_i/\Q(\alpha_i)$ contain transpositions. If $d$ is prime, it follows that for these $c\in  \mathcal{B}_d$, $K_n(\phi_c)/K_{n-1}(\phi_c)$ is maximal for all $n\ge 2$.

In case (2), note that since $c\in  \mathcal{B}_d$, and since the orbit of 0 under $\phi_{(d,c)}$ forms a rigid divisibility sequence, $v_q(\phi_c^n(0))=v_q(d-1)$ for all $n\ge 1$. Thus $(v_q(\phi_c^n(0)),d-1)=1$ for all $n\ge 1$. Fixing any $n\ge 2$, and letting $\alpha_i$ be a root of $\phi_c^{n-1}(x)$, $N_{\Q(\alpha_i)/\Q}(c(d-1)-\alpha_i)=\phi_c^{n-1}(c(d-1))=\phi_c^n(0)$. If $\frak{p}$ is a prime of $\mathcal{O}_{\Q(\alpha_i)}$ lying above $q$, then, $v_{\frak{p}}(c(d-1)-\alpha_i)\mid v_q(\phi^n(0))$. Thus $(v_{\frak{p}}(c(d-1)-\alpha_i),d-1)=1$. We then apply Theorem \ref{thm:2trans} to conclude that for any $n\ge 1$, $\Gal(M_i/\Q(\alpha_i))$ is doubly transitive. When $n>N$, $\Gal(M_i/\Q(\alpha_i))$ contains a transposition; thus, $\Gal(M_i/\Q(\alpha_i))\cong S_d$. From Proposition \ref{prop:max}, it follows that $K_n/K_{n-1}$ is maximal for all $n>N$. Therefore $G_{\Q}(\phi_{(d,c)})\le \tu{Aut}(T_{\infty})$ is a finite index subgroup for all $c\in  \mathcal{B}_d$. As previously, we also observe that for any sufficiently large $c\in  \mathcal{B}_d$, $K_n(\phi_c)/K_{n-1}(\phi_c)$ is maximal for all $n\ge 2$.

Finally, from the discriminant formula in (\ref{disctrinomial}) for $\phi_{(d,c)}$, along with Proposition \ref{prop:LNV} and Theorem \ref{thm:2trans}, one sees that $G_1(\phi_{(d,c)})\cong S_d$ so long as $(d-1)^{d-2}c^d-1$ is not a square in $\Z$ outside of $\tu{Supp}(d)$. But any curve over $\Q$ of the form \begin{equation}{\label{G1diag}}d_iy^2=kx^d-1\end{equation} for $d_i\mid d$ has at most finitely many solutions $(x,y)\in\Z^2$. We have thus shown that for each $n\ge 1$, there are finitely many $c\in  \mathcal{B}_d$ such that $K_n(\phi_c)/K_{n-1}(\phi_c)$ is not maximal. Excluding these parameters, any sufficiently large value of $c\in  \mathcal{B}_d$ must satisfy $G_{\Q}(\phi_c)=\tu{Aut}(T_{\infty})$. This completes the proof of Theorem \ref{thm:abcOdoni}. \end{proof}

 \end{document}